\documentclass[12pt,leqno,twoside]{article}
\usepackage{amssymb}
\usepackage{amsmath}
\usepackage{amsthm}
\usepackage{t1enc}
\usepackage[cp1250]{inputenc}
\usepackage{a4,indentfirst,latexsym}
\usepackage{graphics}
\usepackage{mathrsfs}
\usepackage{cite,enumitem,graphicx}
\usepackage[colorlinks=true,urlcolor=black,
citecolor=black,linkcolor=black,linktocpage,pdfpagelabels,
bookmarksnumbered,bookmarksopen]{hyperref}
\usepackage[colorinlistoftodos]{todonotes}

\input xy
\xyoption{all}

\newtheorem{Thm}{Theorem}[section]

\newtheorem{Lem}[Thm]{Lemma}

\newtheorem{Rem}[Thm]{Remark}

\newcommand\mytop[2]{\genfrac{}{}{0pt}{}{#1}{#2}}

\newcommand{\eps}{\varepsilon}

\def\id{\mathrm{id}}

\def\Z{\mathbb{Z}}

\def\R{\mathbb{R}}
\def\C{\mathbb{C}}

\def\curl{\mathrm{curl}}

\newcommand{\cC}{{\mathcal C}}

\newcommand{\cF}{{\mathcal F}}

\newcommand{\cL}{{\mathcal L}}
\newcommand{\cM}{{\mathcal M}}

\newcommand{\cO}{{\mathcal O}}

\newcommand{\cU}{{\mathcal U}}

\newcommand{\al}{\alpha}
\newcommand{\be}{\beta}
\newcommand{\ga}{\gamma}
\newcommand{\de}{\delta}
\newcommand{\ka}{\kappa}

\newcommand{\om}{\omega}
\newcommand{\si}{\sigma}
\newcommand{\Ga}{\Gamma}
\newcommand{\De}{\Delta}
\newcommand{\La}{\Lambda}
\newcommand{\Om}{\Omega}
\newcommand{\Si}{\Sigma}

\newcommand{\ham}{H_\Omega}
\newcommand{\conf}{\cF_N\Omega}

\newcommand{\pa}{\partial}
\def\id{\mathrm{id}}

\newcommand{\wh}{\widehat}
\newcommand{\wt}{\widetilde}

\newcommand{\beq[1]}{\begin{equation}\label{eq:#1}}
\newcommand{\eeq}{\end{equation}}

\numberwithin{equation}{section}

\begin{document}

\title{Periodic solutions of the N-vortex Hamiltonian system in planar domains}
\author{Thomas Bartsch \and Qianhui Dai}
\date{}
\maketitle

\begin{abstract}
We investigate the existence of collision-free nonconstant periodic solutions of the $N$-vortex problem in domains $\Om\subset\C$. These are solutions $z(t)=(z_1(t),\dots,z_N(t))$ of the first order Hamiltonian system
\[
\dot{z}_k(t)=-i\nabla_{z_k} \ham\big(z(t)\big),\quad k=1,\dots,N,
\]
where the Hamiltonian $\ham$ has the form
\[
\ham(z_1,\dots,z_N)
 = \frac1{2\pi}\sum_{\mytop{j,k=1}{j\ne k}}^N \log\frac1{|z_j-z_k|} - F(z).
\]
The function $F:\Om^N\to\R$ depends on the regular part of the hydrodynamic Green's function and is unbounded from above. The Hamiltonian is unbounded from above and below, it is singular, not integrable, energy surfaces are not compact and not known to be of contact type. We prove the existence of a family of periodic solutions $z^r(t)$, $0<r<r_0$, with arbitrarily small minimal period $T_r\to0$ as $r\to0$. The solutions are close to the singular set of $\ham$. Our result applies in particular to generic bounded domains, which may be simply or multiply connected. It also applies to certain unbounded domains. Depending on the domain there are multiple such families.
\end{abstract}



\section{Introduction}\label{sec:intro}
The dynamics of $N$ point vortices $z_1(t),\dots,z_N(t)$ in a domain $\Om\subset\C$ is governed by a Hamiltonian system
\[
\Ga_k\dot{z}_k(t)=-i\nabla_{z_k} \ham\big(z(t)\big),\quad k=1,\dots,N, \tag{HS}
\]
where $i\in\C$ is the imaginary unit. Here $\Ga_k\in\R$ is the strength of the $k$-th vortex $z_k$ which may be positive or negative according to the orientation of the vortex. The system can be derived from the Euler equations
\beq[euler]
\left\{
\begin{aligned}
v_t+(v\cdot\nabla)v&=-\nabla P\\
\nabla\cdot v&=0
\end{aligned}
\right.
\eeq
which describe the velocity field $v$ and the pressure $P$ of an ideal (i.~e.\ incompressible and non-viscous) fluid in $\Om$. Passing to the equation for the vorticity $\om=\nabla\times v=\pa_1v_2-\pa_2v_1$,
\beq[euler-vort]
\om_t+v\cdot\nabla\om=0,
\eeq
and making a point vortex ansatz $\om=\sum_{k=1}^{N}\Ga_k\de_{z_k}$ where $\de_{z_k}$ is the usual Dirac delta, the point vortices $z_k(t)$ move according to (HS) with a special Hamiltonan $\ham$. This goes back to Kirchhoff \cite{kirchhoff:1876}, Routh \cite{routh:1881} and Lin \cite{lin:1941a, lin:1941b}; see \cite{flucher-gustafsson:1997, majda-bertozzi:2001, marchioro-pulvirenti:1994, newton:2001, saffman:1992} for modern treatments of vorticity methods.

If $\Om=\C$ is the plane then the Hamiltonian $H_\C$ is the Kirchhoff-Routh path function
\[
H_\C(z) = \frac1{2\pi}\sum_{\mytop{j,k=1}{j\ne k}}^N \Ga_j\Ga_k\log\frac1{|z_j-z_k|}\,.
\]
If $\Om\ne\C$ is a domain one has to take the influence of the boundary into account. In that case the Hamiltonian has the form
\beq[ham-gen]
\ham(z) = \frac1{2\pi}\sum_{\mytop{j,k=1}{j\ne k}}^N \Ga_j\Ga_k\log\frac1{|z_j-z_k|} - F(z)
\eeq
where $F:\Om^N\to\R$ is smooth. In order to describe it in the case of a  domain with solid boundary let $G$ be a hydrodynamic Green's function (see \cite{flucher-gustafsson:1997}) in $\Om$ with regular part $g$, so
\[
G(w,z) = \frac1{2\pi}\log\frac1{|w-z|} - g(w,z)\qquad\text{for }w,z\in\Om,\ w\ne z.
\]
The domain $\Om\ne\C$ may be bounded or unbounded; in the unbounded case conditions on the behavior of the Green's function at infinity have to be assumed to make it unique. If $\Om$ is bounded and simply connected then $G$ is the Green's function for the Dirichlet Laplacian. The leading term of the regular part of $G$, i.~e.\ the function
\[
h:\Om\to\R,\quad h(z)=g(z,z),
\]
is the hydrodynamic Robin function. Our sign conventions here imply that $g$ is bounded below and $h(z)\to\infty$ as $z\to\pa\Om$. In particular $h$ achieves its minimum in $\Om$ if $\Om$ is bounded. The Hamiltonian is
\[
\ham(z) = \sum_{\mytop{j,k=1}{j\ne k}}^N \Ga_j\Ga_k G(z_j,z_k)
           - \sum_{k=1}^N\Ga_k^2h(z_k),
\]
hence
\[
F(z) = \sum_{\mytop{j,k=1}{j\ne k}}^N \Ga_j\Ga_k g(z_j,z_k) + \sum_{k=1}^N\Ga_k^2h(z_k)
\]
in \eqref{eq:ham-gen}.
The dynamics of a single vortex in $\Om$ is completely described by the Robin function because $h$ coincides with the Hamiltonian in that case; see \cite{gustafsson:1979}. Our main theorem shows that the Robin function also plays a fundamental role in the analysis of the dynamics of $N\ge2$ point vortices; see Theorem~\ref{thm:main} and Remark~\ref{rem:superposition}.

There are many results about special solutions of (HS) if $\Om=\C$ is the whole plane. We refer to the monograph \cite{newton:2001}, the survey article \cite{aref-etal:2002} on vortex crystals, and the references therein. In \cite{newton:2001} one can also find an introduction to the analysis of the point vortex flow in domains. The majority of the literature deals with special domains and geometries like vortices in corners or channels, above flat walls or in a semidisk.

In this paper we consider the problem whether (HS) has nonconstant periodic solutions in a domain $\Om\ne\C$. This is, of course, a basic question about any Hamiltonian system, which however has not been addressed for the $N$-vortex problem in a general domain. The difficulty is that the Hamiltonian is singular, not integrable, and energy levels are not compact and not known to be of contact type, so standard methods do not apply. It is even difficult to prove the existence of stationary points of (HS). The only exception is the case $N=2$ and $\Ga_1\Ga_2<0$ when $\ham$ is bounded above and $\ham(z_1,z_2)\to-\infty$ if $z_k\to\pa\Om$ or $z_1-z_2\to0$. Thus energy surfaces are compact, and periodic solutions abound according to a result of Struwe \cite{struwe:1990}. In all other cases energy surfaces are not compact, and $\ham$ is not bounded from above or below, in fact $\ham(z)$ may approach any value in $\R\cup\{\pm\infty\}$ if some of the $z_k$'s approach the boundary $\pa\Om$. Therefore it is not surprising that there are no results on the existence of nonconstant periodic solutions except when the domain is radial. In that case, and when $\Ga_k=\Ga_1$ for all $k$, it is not difficult to find periodic solutions where the $N$ vortices are arranged symmetrically.

If $\Om$ is bounded, not simply connected, and if all $\Ga_k=1$ then the existence of a critical point of $\ham$ has been proved by del Pino, Kowalczyk and Musso in \cite{delpino-etal:2005}. In an arbitrary bounded domain, if $2\le N\le4$, if the $\Ga_k$'s have alternating signs and satisfy additional conditions, the existence of a critical point of $\ham$ has been proved in \cite{bartsch-pistoia:2014} improving the earlier result in \cite{bartsch-pistoia-weth:2010} for the case $\Ga_k=(-1)^k$. If the domain has an axis of symmetry, a critical point of $\ham$ has been found in \cite{bartsch-pistoia-weth:2010} for arbitrary $N\ge 2$. For a nonsymmetric domain the existence of critical points of $\ham$ when $N\ge5$ is unknown, whatever values the $\Ga_k$'s take.

In the present paper we treat the case when all $\Ga_k$ are the same, without loss of generality $\Ga_k=1$. Our main result Theorem~\ref{thm:main} states the existence of periodic solutions with arbitrarily small minimal periods where the $N$ vortices oscillate around a stable critical point of the Robin function $h$. These solutions are far away from the equilibrium solution found in \cite{delpino-etal:2005} in the case when $\Om$ is not simply connected. For bounded domains the existence of a stable critical point of the Robin function is generic with respect to perturbations of the domain as has been proved recently by Micheletti and Pistoia \cite{micheletti-pistoia:2014}; see Remark~\ref{rem:robin} below. Thus for a generic bounded domain our result applies. It may be worthwhile to mention that our result also applies to unbounded domains, and can be extended to the point vortex flow on surfaces.

There is a large literature on periodic solutions of singular Hamiltonian systems, but most papers deal with second order systems of the $N$-body type. Periodic solutions for first order singular Hamiltonian systems have been investigated by Carminati, S\'er\'e, Tanaka in \cite{carminati-sere-tanaka:2006} and \cite{tanaka:1996}. In \cite{tanaka:1996} the existence of periodic solutions with fixed period has been proved for a non-autonomous Hamiltonian $H(t,q,p)$ which is $2\pi$-periodic in $t$, and which has the form
$H(t,q,p) \sim \frac1\be|p|^\be-\frac1{|q|^\al}$ with $p,q\in\R^N$, $\al\ge\be>1$.
Thus the singularity is at $q=0$. Clearly the behavior of $H$ with respect to the conjugate variables $p$ and $q$ is completely different from the class of Hamiltonians we consider here. Moreover, our singular set is much more complicated than the one in \cite{tanaka:1996}. The same applies to \cite{carminati-sere-tanaka:2006} where periodic solutions on a fixed energy surface have been found. The energy surface has to be of contact type, and the existence is obtained by reduction to a theorem of Hofer and Viterbo \cite{hofer-viterbo:1988} on the Weinstein conjecture in cotangent bundles of manifolds. Moreover, in \cite{carminati-sere-tanaka:2006,tanaka:1996} the behavior of $H$ near the singularity is modeled after the "strong force" condition from \cite{gordon:1975} for second order Hamiltonian systems.

Neither the results nor the techniques of the existing papers on singular Hamiltonian systems apply to Hamiltonians
\[
\ham:\conf=\{z\in\Om^N: z_j\ne z_k \text{ for }j\ne k\}\to \R
\]
of the form \eqref{eq:ham-gen}. In addition to the well-known technical problems due to the strong indefiniteness of the action functional for $T$-periodic solutions
\[
J(z)=\frac12\sum_{k=1}^N \int_0^T(i\dot z_k)\cdot z_k\,dt - \int_0^{T} \ham(z)\,dt
\]
new difficulties arise. The first integral in the action functional is defined on $H^{1/2}(\R/T\Z,\R^{2N})$, whereas the second integral prefers $z(t)\in\conf$. Since $H^{1/2}(\R/T\Z,\R^{2N})$ does not embed into $L^\infty$ the condition $z(t)\in\conf$ does not define an open subset of $H^{1/2}(\R/T\Z,\R^{2N})$. Working in $H^1(\R/T\Z,\R^{2N})$, or other spaces which embed into $L^\infty$, will cause compactness problems. Compactness problems appear anyway because there is no definite behavior of $H(z)$ as $z\to\pa\conf\subset\C^N$.

\section{Statement of results}\label{sec:results}
Let $g:\overline{\Om}\times\Om\to\R$ be of class $\cC^2$ and symmetric: $g(w,z)=g(z,w)$ for all $w,z\in\Om$. For instance, $g$ may be the regular part of a hydrodynamic Green's function on $\Om\subset\C$.  We consider the Hamiltonian system (HS) with
\[
\ham(z)
 = \frac1{2\pi}\sum_{\mytop{j,k=1}{j\ne k}}^N\log\frac1{|z_j-z_k|}
    - \sum_{\mytop{j,k=1}{j\ne k}}^N g(z_j,z_k) - \sum_{k=1}^N g(z_k,z_k).
\]
As in Section~\ref{sec:intro} we define the "Robin function"
\[
h:\Om\to\R,\quad h(z)=g(z,z).
\]
A critical point $a\in\Om$ of $h$ with $h(a)=c$ is said to be stable if it is isolated and its critical group $H_*(h^c,h^c\setminus\{a\})$ is not trivial. Here
$h^c=\{z\in\Om:h(z)\le c\}$ is the usual sublevel set, and $H_*$ denotes any kind of homology theory; cohomology serves as well. An isolated local minimum or maximum is stable as is a nondegenerate saddle point. If $h(a+z)=h(a)+\al\text{Re}(z^k)+\be\text{Im}(z^k)+o(|z|^k)$ as $z\to0$ for some $k\ge2$ and $\al^2+\be^2\ne0$ then $a$ is stable but degenerate.

\begin{Thm}\label{thm:main}
If $a_0\in\Omega$ is a stable critical point of $h$, then there exists $r_0>0$, such that for each $0<r\leq r_0$, (HS) has a periodic solution $z^r=(z^r_1,\dots,z^r_N)$ with minimal period $T_r=4\pi^2r^2/(N-1)$ such that $z^r_k(t)=z^r_1(t+(k-1)T_r/N)$ for every $k=1,\dots,N$. In the limit $r\to0$ the vortices $z^r_k$ move on circles in the following sense. There exists $a_r\in\Om$ with $a_r\to a_0$ such that the rescaled function
\[
u^r_1(t):=\frac1r\big(z^r_1(T_rt/2\pi)-a_r\big)
\]
satisfies
\[
u^r_1(t) \to u_0(t) := e^{it}.
\]
The convergence $u^r_1\to u_0$ as $r\to0$ holds in $H^1(\R/2\pi\Z,\C)$.
\end{Thm}

\begin{Rem}\label{rem:robin}
\rm
a) Since a hydrodynamic Robin functin satisfies $h(z)\to\infty$ as $z\to\pa\Om$, the minimum is always achieved in a bounded domain. Caffarelli and Friedman \cite{caffarelli-friedman:1985} showed that the Robin function is strictly convex if $\Om$ is convex but not an infinite strip. In the latter case the function is still convex and explicitely known (see \cite{bandle-flucher:1996}), but of course invariant under translations, so that it cannot have an isolated critical point. Thus in a bounded convex domain the Robin function has a unique critical point, the global minimum. This is in fact nondegenerate according to \cite[Theorem~3.1]{caffarelli-friedman:1985}. If the domain is smooth, bounded, symmetric with respect to the origin, and convex in the direction of the two coordinates then Grossi \cite{grossi:2002} showed that the origin is a nondegenerate critical point. For a generic bounded smooth domain, Micheletti and Pistoia \cite{micheletti-pistoia:2014} proved that all critical points of the Robin function are nondegenerate. More precisely, they considered an arbitrary bounded smooth domain $\Om\subset\R^n$ and diffeomorphisms of $\R^n$ of the form $\id+\Theta$ with $\Theta:\R^n\to\R^n$ small in the $\cC^k$-norm. They showed that the Robin function of the Dirichlet Laplacian of $(\id+\Theta)(\Om)$ has only nondegenerate critical points for a residual set of $\Theta$'s. Thus in a generic domain Theorem~\ref{thm:main} applies and yields periodic solutions with arbitrarily small minimal period oscillating around the minimum of the Robin function.

b) Theorem~\ref{thm:main} does not apply to the annulus due to its rotational symmetry, because the minimum of $h$ is not isolated but there is a circle of minima. On the other hand, perturbing the annulus one obtains domains where the Robin function has at least two critical points, a minimum and a saddle point. One can also construct simply connected domains, e.~g.\ dumb-bell shaped, where the Robin function has arbitrarily many local minima, and many saddle points; see \cite{flucher:1992}. The function $r(z) = e^{-h(z)}$ is the inner radius (conformal radius for simply connected domains) from the theory of complex functions; see \cite{haegi:1950} where one can find a discussion of the geometric role of critical points of $r$, hence of $h$.
\end{Rem}

\begin{Rem}\label{rem:superposition}
\rm
a) One may consider Theorem~\ref{thm:main} as a kind of singular Lyapunov center theorem.

b) The solutions obtained in Theorem~\ref{thm:main} are close to $v+w^r$ where $v(t)=(a_0,\dots,a_0)$ and $w^r(t)=(w^r_1,\dots,w^r_N)$ with
$w^r_k(t)=re^{2\pi(k-1)i/N}e^{2\pi it/T_r}$. Observe that $v_k=a_0$ is a stationary solution of
\beq[robinflow]
\dot v_0 = -i\nabla h(v_0).
\eeq
We believe that given any solution $v_0(t)$ of \eqref{eq:robinflow} and setting $v=(v_0,\dots,v_0)$, for $r>0$ small there are solutions $z^r(t)$ of (HS) which are close on finite time intervals to the superposition $v+w^r$ as in Theorem~\ref{thm:main}. It is tempting to conjecture that one can obtain periodic and quasiperiodic solutions of (HS) by starting with a periodic solution $v_0$ of \eqref{eq:robinflow} and superpose $w^r$ for $r>0$ small.
\end{Rem}

\begin{Rem}\label{rem:regular}
\rm
A very interesting and challenging problem is to regularize the periodic solution which we found here. Given an equilibrium solution $z=(z_1,\dots,z_N)$ of (HS) in a bounded simply-connected smooth domain $\Om$, Cao, Liu and Wei \cite{cao-liu-wei:2013,cao-liu-wei:2014} construct a family of smooth stationary solutions $v_\eps$ of the Euler equations \eqref{eq:euler} such that its vorticies
$\om_\eps = \curl\, v_\eps$ converge as $\eps\to0$ towards the stationary point vortex
solution $\om = \sum_{k=1}^{N}\Ga_k\de_{z_k}$ of \eqref{eq:euler-vort}. The vorticities $\om_\eps$ have support in shrinking neighborhoods of the points $z_k$. This improves the earlier regularization result in the case $N=2$ due to Smets and van Schaftingen \cite{smets-schaftingen:2012}. These papers are based on the method of stream functions. Another way of numerically regularizing point vortex solutions is the vortex patch method for which we refer to \cite{marchioro-pulvirenti:1994,saffman:1992}. We are not aware of results about regularizing a periodic point vortex solution of (HS) to a periodic solution of \eqref{eq:euler}, \eqref{eq:euler-vort}.
\end{Rem}

\section{Preliminaries}\label{sec:prelim}
Without loss of generality we assume $a_0=0$. The function
\[
F:\Om^N\to\R,\quad F(z)=\sum_{k=1}^N h(z_k)+\sum_{\mytop{j,k=1}{j\ne k}}^Ng(z_j,z_k),
\]
satisfies
\begin{equation}\label{eq:F-sym}
F(z_1,\dots,z_N)=F(z_{\si(1)},\dots,z_{\si(N)})
\end{equation}
for any permutation $\si\in\Si_N$ of $\{1,\dots,N\}$. We rescale the problem by setting
\[
\begin{aligned}
H_r(u)
 &:=\frac{T_r}{2\pi r^2}\left(H_{\Om}(ru)
      + \frac1{2\pi}\sum_{\mytop{j,k=1}{j\ne k}}^N\log|r| \right)\\
 & = \frac1{N-1}\sum_{\mytop{j,k=1}{j\ne k}}^N\frac1{\log|u_j-u_k|}
      - \frac{2\pi}{N-1} F(ru).
\end{aligned}
\]
Then $z$ is a $T_r$-periodic solution of (HS) if and only if
$u(t):=\frac{1}{r}z\left(\frac{T_r}{2\pi}t\right)$ is a $2\pi$-periodic
solution of
\[
\dot{u_k}= -i \nabla_{u_k} H_r(u),\quad k=1,\cdots,N.
\tag{${\textrm{HS}}_r$}
\]
Observe that $H_r$ defines a function
\[
H: \cO:=\big\{(r,u)\in \R\times\C^N: u_j\neq u_k\ \text{for }j\neq k,\
         ru_k\in\Omega\ \text{for all }k\big\} \to \R
\]
which is also defined for $r=0$.

Let $L^2_{2\pi}(\C^N)=L^2(\R/2\pi\Z,\C^N)$ be the Hilbert space of $2\pi$-periodic square integrable functions with scalar product
$$
\langle x,y\rangle_{L^2}
 = \int_0^{2\pi}\big\langle x(t),y(t)\big\rangle_{\R^{2N}}\,dt
 = \sum_{k=1}^N\int_0^{2\pi}\text{Re}(\overline{x_k(t)}y_k(t))\,dt
$$
and associated norm $\|\cdot\|_{L^2}$. The space $H^1_{2\pi}(\C^N)=H^1(\R/2\pi\Z,\C^N)$ is the Sobolev space of $2\pi$-periodic functions which are absolutely continuous with square integrable derivative with scalar product
$$
\langle x,y\rangle = \langle x,y\rangle_{L^2}+\langle\dot{x},\dot{y}\rangle_{L^2}
$$
and associated norm $\|\cdot\|$. Recall the action of $S^1=\R/2\pi\Z$ on $L^2_{2\pi}(\C^N)$ and $H^1_{2\pi}(\C^N)$ given by time shift, and the action of $\Si_N$ which permutes the components. These combine to yield an isometric action of $S^1\times\Si_N$ given by
\[
(\theta,\si)*u(t) = \left(u_{\si^{-1}(1)}(t+\theta),\dots,u_{\si^{-1}(N)}(t+\theta)\right)
\]
for $u\in L^2_{2\pi}(\C^N)$ or $H^1_{2\pi}(\C^N)$ and $(\theta,\si)\in S^1\times\Si_N$. We also use the notation $\theta*u$ and $\si*u$ for $\theta\in S^1\subset S^1\times\Si_N$, $\si\in\Si_N\subset S^1\times\Si_N$. The action functional corresponding to $(\textrm{HS}_r)$ is given by
$$
\Phi_r(u)
 := \frac12\sum_{k=1}^N \int_0^{2\pi} \langle i\dot{u}_k,u_k\rangle_{\R^{2}}\,dt
     - \int_0^{2\pi}H_r(u)\,dt.
$$
Observe that $\Phi(r,u)=\Phi_r(u)$ is defined for $(r,u)$ in the set
\[
\La
 := \big\{(r,u)\in\R\times H^1_{2\pi}(\C^N): (r,u(t))\in\cO\ \text{for all }t\big\}
\]
which is an open subset of $\R\times H^1_{2\pi}(\C^N)$. Critical points of $\Phi_r$ for $r>0$ correspond to $2\pi$-periodic solutions of $(\textrm{HS}_r)$. Clearly $\La$ and $\Phi_r$ are invariant under the action of $S^1\times\Si_N$.

For $r=0$ there holds
\[
H_0(u) = \frac{1}{N-1}\sum_{\mytop{j,k=1}{j\neq k}}^N\frac1{\log|u_j-u_k|}
          - \frac{2\pi}{N-1}F(0),
\]
hence system $(\textrm{HS}_0)$ is given by
\[
\dot{u_k} = \frac{2i}{N-1}\sum_{\mytop{j=1}{j\neq k}}^N\frac{u_j-u_k}{|u_j-u_k|^2},\quad k=1,\cdots,N.
\tag{$\textrm{HS}_0$}
\]
This system has a family of $2\pi$-periodic solutions $\theta*U_a$ parametrized by $\theta\in S^1$ and $a\in\C$ where
\[
U_a(t) = \left(
\begin{array}{c}
a+u_0(t) \\
a+u_0(t+\frac{2\pi}{N})\\
\vdots\\
a+u_0(t+\frac{2\pi(N-1)}{N})
\end{array}
\right).
\]
The set
\[
\cM = \big\{\theta* U_a: \theta\in S^1, a\in\C\big\}
\]
is a 3-dimensional non-compact submanifold of $H^1_{2\pi}(\C^N)$ consisting of $2\pi$-periodic solutions of $(\textmd{HS}_0)$.

Let $\si=(1\ 2\ \dots\ N)\in\Si_N$ be the right shift, and set
$\tau := \left(\frac{2\pi}{N},\si^{-1}\right)\in S^1\times\Si_N$, hence
\begin{equation*}
\tau*u(t) = \left(
\begin{array}{c}
u_N(t+\frac{2\pi}{N}) \\
u_1(t+\frac{2\pi}{N})\\
\vdots\\
u_{N-1}(t+\frac{2\pi}{N})
\end{array}
 \right).
\end{equation*}
Obviously $\langle\tau\rangle\subset S^1\times\Si_N$ is a cyclic subgroup of $S^1\times\Si_N$ of order $N$. Since $\Phi_r$ is $(S^1\times\Si_N)$-invariant, by the principle of symmetric criticality it is sufficient to find critical points of $\Phi_r$ constrained to
\[
\La^\tau=\{(r,u)\in\La: u=\tau*u\}.
\]
Clearly for $(r,u)\in\La$ we have $(r,u)\in \La^\tau$ if, and only if,
$u_k(t)=u_1\left(t+\frac{2\pi(k-1)}{N}\right)$ for all $k=1,\dots,N$. Thus the map
\[
H^1_{2\pi}(\C) \to H^1_{2\pi}(\C^N),\quad
v \mapsto \wh{v} := \left(v,\frac{2\pi}{N}*v,\dots,\frac{2\pi(N-1)}{N}*v\right)^{tr},
\]
induces a diffeomorphism
\[
\cM_1:=\{\theta*u_a: \theta\in S^1, u_a=u_0+a, a\in\C\} \to \cM \subset \La^\tau,
\]
and a diffeomorphism
\[
\La_1 := \{(r,u_1)\in\R\times H^1_{2\pi}(\C):(r,\wh{u}_1)\in\Lambda^\tau\} \to \La^\tau,
\quad (r,u_1)\mapsto (r,\wh{u}_1).
\]
Defining $\Psi:\La_1 \to \R$ by $\Psi(r,u_1) = \Psi_r(u_1) := \Phi_r(\wh{u}_1)$, i.~e.\
\[
\begin{aligned}
\Psi_r(u_1)
& = \frac{N}2 \int_0^{2\pi} \langle i\dot{u_1},u_1\rangle_{\R^{2}}\,dt
     + \frac{N}{N-1} \sum_{k=1}^{N-1}
        \int_0^{2\pi}\log\left|u_1-\frac{2k\pi}{N}*u_1\right|\,dt\\
&\hspace{1cm}
     + \frac{2\pi}{N-1}\int_0^{2\pi}F(r\wh{u}_1)\,dt,
\end{aligned}
\]
it suffices to find critical points of $\Psi_r$. More precisely, if $u_1$ is a critical point of $\Psi_r$ then $\wh{u}_1$ is a critical point of $\Phi_r$. A straightforward computation shows that
$$
\nabla\Psi_r(u_1)
 = N(Id-\De)^{-1}\left(i\dot{u}_1
    + \frac2{N-1}\sum_{k=1}^{N-1}
       \frac{u_1-\frac{2k\pi}{N}*u_1}{|u_1-\frac{2k\pi}{N}*u_1|^2}
    + \frac{2\pi r}{N-1}\pa_1F(r\wh{u}_1)\right),
$$
where $\De: H^2_{2\pi}(\C)\to L^2_{2\pi}(\C)$, $\De v=\ddot v$, and $\pa_1$ means the gradient in the real sense with respect to the first complex component.

Finally we fix $\de>0$ such that the $\delta$-neighborhood
\beq[def-de]
\cU_\de(\cM_1) \subset \La_0 :=
 \big\{u_1\in H^1_{2\pi}(\C):
   u_1(t)\ne u_1(t+\frac{2k\pi}{N})\text{ for all } t,\text{ all }k=1,\dots,N-1\big\}
\eeq
of $\cM_1$ in $H^1_{2\pi}(\C)$ is contained in the domain $\La_0$ of $\Psi_0$. This is possible because $H^1_{2\pi}(\C)$ imbeds into $C^0_{2\pi}(\C)$.

\section{Finite-Dimensional Reduction}\label{sec:reduction}
Since the action functional $\Phi_r$ is strongly indefinite it is easier to make a reduction to a finite-dimensional variational problem first. Recall that
\[
\Psi_0(u_1)
 = \frac{4\pi^2}{N-1} F(0)+\frac{N}2\int_0^{2\pi}\langle i\dot{u}_1,u_1\rangle_{\R^2} \,dt
    +\frac{N}{N-1} \sum_{k=1}^{N-1}
     \int_0^{2\pi} \log\left|u_1-\frac{2k\pi}{N}*u_1\right|\,dt
\]
for $u_1\in\La_0$, hence
\[
\nabla \Psi_0(u_1)
 = N(Id-\De)^{-1} \left(i\dot{u}_1\frac2{N-1}
    \sum_{k=1}^{N-1} \frac{u_1-\frac{2k\pi}{N}*u_1}{|u_1-\frac{2k\pi}{N}*u_1|^2}\right)
\]
and
\[
\begin{aligned}
\nabla^2\Psi_0(u_1)[v]
 &= N(Id-\Delta)^{-1}\Bigg(i\dot{v}
     + \frac2{N-1} \sum_{k=1}^{N-1}
       \frac{v-\frac{2k\pi}{N}*v}{|u_1-\frac{2k\pi}{N}*u_1|^2}\\
 &\hspace{1cm}
    - \frac{4}{N-1}\sum_{k=1}^{N-1}
       \frac{\big\langle u_1-\frac{2k\pi}{N}*u_1,v-\frac{2k\pi}{N}*v\big\rangle_{\R^2}}
        {|u_1-\frac{2k\pi}{N}*u_1|^4} \left(u_1-\frac{2k\pi}{N}*u_1\right)\Bigg)
\end{aligned}
\]
Clearly we have
\beq[Psi_0]
\nabla \Psi_0(u_1+a)=\nabla \Psi_0(u_1)\quad\text{and}\quad
 \nabla^2\Psi_0(u_1+a)=\nabla^2 \Psi_0(u_1)
\eeq
for any $u_1\in \Lambda_0$, any $a\in\C$. A direct computation shows for $u_0(t)=e^{it}$ that
\[
\nabla^2\Psi_0(u_0)[v]
 = N(Id-\Delta)^{-1}\left(i\dot{v} - \frac2{N-1} \sum_{k=1}^{N-1}
    \frac{(u_0-\frac{2k\pi}{N}*u_0)^2}{|1-e^{2k\pi i/N}|^4}
     \left(\overline{v}-\frac{2k\pi}{N}*\overline{v}\right)\right).
\]

\begin{Lem}\label{lem:M-nondeg}
$\cM_1$ is a nondegenerate critical manifold of $\Psi_0$, in particular $\textrm{Ker}\,\nabla^2\Psi_0(u_0) = T_{u_0}\cM_1$.
\end{Lem}

\begin{proof}
Since $\cM_1 = S^1*(u_0+\C)$ is the homogeneous space obtained from $u_0$ via the translations $u_0\mapsto u_0+a$ and via the $S^1$-action, and since $\Psi_0$ is invariant under these actions, it is sufficient to show that
$\textrm{Ker}\,\nabla^2\Psi_0(u_0) = T_{u_0}\cM_1$. Clearly
$T_{u_0}\cM_1\subset \textrm{Ker}\,\nabla^2\Psi_0(u_0)$, hence we only need to prove that
\[
\textrm{Ker}\,\nabla^2\Psi_0(u_0) \subset T_{u_0}\cM_1
 = \left\{ a+ic\cdot u_0\,:\,a\in \C,\,c\in \R\,\,\right\}.
\]

Consider an element  $v\in\textrm{Ker}\,\nabla^2\Psi_0(u_0)$, so that
\beq[v-kernel]
i\dot{v}(t) - \frac2{N-1} \sum_{k=1}^{N-1}
  \frac{(e^{it}-e^{i(t+\frac{2k\pi}{N})})^2}{|1-e^{i\cdot\frac{2k\pi}{N}}|^4}
  \left(\overline{v(t)}-\overline{v(t+\frac{2k\pi}{N})}\right)
  = 0.
\eeq
We write $v$ in its Fourier expansion, $v(t)=\sum_{n\in\Z}\al_n e^{int}$ with coefficients
$\al_n\in\C$, and substitute it into \eqref{eq:v-kernel} obtaining
\[
\sum_{n\in\Z}n\al_n e^{int}
 + \frac2{N-1}\sum_{n\in\Z} \sum_{k=1}^{N-1}
    \frac{1-e^{-i\cdot\frac{2kn\pi}{N}}}{(1-e^{-i\cdot\frac{2k\pi}{N}})^2}\cdot
    \overline{\al}_ne^{i(2-n)t}=0.
\]
Setting
\[
\xi_n:=\sum_{k=1}^{N-1}\frac{1-e^{i\cdot\frac{2\pi k(n-2)}{N}}}
{(1-e^{-i\cdot\frac{2\pi k}{N}})^2}
\]
a comparison of the coefficients yields for each $n\in\Z$\,:
\beq[coeff]
n\al_n+\frac{2}{N-1}\xi_n\cdot\overline{\al}_{2-n} = 0,
\eeq
and, replacing $n$ by $2-n$:
\beq[coeff-2]
(2-n)\overline{\al}_{2-n}+\frac{2}{N-1}\overline{\xi}_{2-n}\cdot\al_n = 0.
\eeq
Observe that
\[
\overline{\xi_n}
 = \sum_{k=1}^{N-1} \frac{1-e^{-i\cdot\frac{2k(n-2)\pi}{N}}}
     {(1-e^{i\cdot\frac{2k\pi}{N}})^2}
 = \sum_{k=1}^{N-1} \frac{1-e^{i\cdot\frac{2(N-k)(n-2)\pi}{N}}}
    {(1-e^{-i\cdot\frac{2(N-k)\pi}{N}})^2}
 = \sum_{k=1}^{N-1}\frac{1-e^{i\cdot\frac{2k(n-2)\pi}{N}}}
    {(1-e^{-i\cdot\frac{2k\pi}{N}})^2}
 = \xi_n,
\]
hence $\xi_n\in\R$. On the other hand, the computation
\[
\begin{aligned}
\xi_n-\xi_{2-n}
 &= \sum_{k=1}^{N-1}
     \frac{1-e^{i\cdot\frac{2k(n-2)\pi}{N}}}{(1-e^{-i\cdot\frac{2k\pi}{N}})^2}
      - \frac{1-e^{i\cdot\frac{2kn\pi}{N}}}{(1-e^{i\cdot\frac{2k\pi}{N}})^2}\\
 &= \sum_{k=1}^{N-1}
     \frac{e^{i\cdot\frac{4k\pi}{N}}-e^{-i\cdot\frac{4k\pi}{N}}+2e^{-i\cdot\frac{2k\pi}{N}}
      -2e^{i\cdot\frac{2k\pi}{N}}}{|1-e^{i\cdot\frac{2k\pi}{N}}|^4}
\end{aligned}
\]
shows that $\xi_n-\xi_{2-n} \in i\R$ and therefore $\xi_n = \xi_{2-n} \in \R$. Combining this with \eqref{eq:coeff} and \eqref{eq:coeff-2} we deduce
\[
n(2-n)\al_n=\frac{4}{(N-1)^2}\,\xi_n^2\,\al_n \quad\text{for all $n\in\Z$\,,}
\]
which immediately implies
\beq[alpha_n]
\al_n=0,\quad \text{for all }n\ne 0,1,2.
\eeq
Next we take $n=1$ in \eqref{eq:coeff} and obtain, using the equality $\xi_1=\frac{N-1}2$\,:
\beq[alpha_1]
0 = \al_1+\frac2{N-1}\xi_1\cdot\overline{\al}_1 = \al_1+\overline{\al}_1,\quad
\text{thus $\alpha_1\in i\R$\,.}
\eeq
Finally, considering $n=0$ in \eqref{eq:coeff-2} yields
\beq[alpha_2]
\alpha_2=0
\eeq
because $\xi_2=0$. Now \eqref{eq:alpha_n}-\eqref{eq:alpha_2} imply
$v\in T_{u_0}\cM_1$.
\end{proof}

For a given $v\in\cM_1$ we denote $P_v:H^1_{2\pi}(\C)\to T_v\cM_1$ the orthogonal projection. Since $\nabla^2 \Psi_0(v)$ is self-adjoint, $H^1_{2\pi}(\C)$ decomposes into the orthogonal direct sum of $T_v\cM_1=\textrm{Ker}\,\nabla^2 \Psi_0(v)0\textrm{Ran}\,P_v$ and
$N_v\cM_1=\textrm{Ran}\,\nabla^2 \Psi_0(v)=\textrm{Ker}\,P_v$. The equation $(\textrm{HS}_r)$ is equivalent to the system
\[
\left\{
\begin{aligned}
P_v\big(\nabla \Psi_r(u_1)\big)&=0,\\
(Id-P_v)\big(\nabla \Psi_r(u_1)\big)&=0.
\end{aligned}
\right.
\]
We try to find solutions of the form $u_1=\,v+w$ with $v\in\cM_1$ and $w\in N_v\cM_1$ small. Technically, we apply a Lyapunov-Schmidt reduction to the system
\beq[LS2]
\left\{
\begin{aligned}
P_v\big(\nabla \Psi_r(v+w)\big)=0,\\
(Id-P_v)\big(\nabla \Psi_r(v+w)\big)=0.
\end{aligned}
\right.
\eeq
More precisely, for fixed $v\in\cM_1$ and $r\sim0$ we first solve the second equation in \eqref{eq:LS2}, using the contraction mapping principle in a suitable neighborhood of
$0\in N_v\cM_1$. This yields a solution $w=W(r,v)\in N_v\cM_1$ which in turn will be substituted into the first equation of \eqref{eq:LS2}. In order to do this, we fix a constant $\rho>0$ such that $B_{2\rho}(0)\subset\Om$. Then $\Psi_r(u_1)$ is well-defined provided $u_1\in\cU_\de(\cM_1)$ and $|ru_1(t)|\le2\rho$ for all $t$; here $\de$ is from \eqref{eq:def-de}.

\begin{Lem} There exists a constant $r_0=r_0(\rho,\de)>0$ and an $S^1$-equivariant map
\[
W: \cU := \big\{(r,v)\in\R\times\cM_1:|r|\le r_0,v=\theta*u_a,a\in\C,|ra|\le\rho \big\}
 \to H^1_{2\pi}(\C)
\]
such that $W(r,v)\in N_v\cM_1$, satisfying $\|W(r,v)\|\le\de$ and solving the equation
\[
(Id-P_v)\Big(\nabla\Psi_r\big(v+W(r,v)\big)\Big)=0.
\]
\end{Lem}

\begin{proof} The proof is based on an application of the contraction mapping principle, and consists of four steps.

{\sc Step 1.}\quad Reduction to a fixed point problem\\
We fix $R>0$ and define
\[
\cM_1^R:=\big\{v=\theta*u_a\in \cM_1: \theta\in S^1,\,a\in\C,\,|a|\le R\big\}.
\]
We shall define $W(r,v)\in N_v\cM_1$ for $|r|$ small and $v\in\cM_1^R$. First of all, there exists a constant $\ka>0$ such that
\[
\|u\|_{\cC^0}\le\ka\|u\|,\quad \text{for all }u\in H^1_{2\pi}(\C).
\]
Given $r\in\R$ with $|r|\le\frac{\rho}{R}$ and $|r|\le r_1:=\frac{\rho}{1+\ka\de}$, it follows for any $v=\theta*u_a\in\cM_1^R$ and $\|w\|\le\de$ that $v+w\in\cU_\de(\cM_1)$ and
\[
|rv(t)+rw(t)| \le |ra|+|ru_0(t)|+|rw(t)| \le |r|R+|r|+|r|\ka\de \le2 \rho,
\]
for all $t$ so that $\Psi_r(v+w)$ is well-defined.

The second equation in \eqref{eq:LS2} is equivalent to
\[
\begin{aligned}
&(Id-P_v)\circ \nabla^2\Psi_0(v)[w]\\
&\hspace{1cm}
 = -(Id-P_v)\left(\nabla\Psi_0(v+w)-\nabla^2\Psi_0(v)[w]
  +\frac{2\pi rN}{N-1}(Id-\Delta)^{-1}\big( \partial_1F(r\wh{v}+r\wh{w})\big)\right).
\end{aligned}
\]
As a consequence of Lemma~\ref{lem:M-nondeg} the operator
$\cL_v:=(Id-P_v)\circ\nabla^2\Psi_0(v)$ induces an isomorphism $\cL_v|_{N_v\cM_1}$ on $N_v\cM_1$. Also notice that if $v=\theta*u_a$ then
$\cL_v=(Id-P_{\theta*u_0})\circ\nabla^2\Psi_0(\theta*u_0)$ is actually independent of $a$. Thus $(\cL_v|_{N_v\cM_1})^{-1}$ exists and there is a constant $\ga>0$ independent of $v$, such that
\beq[norm-Lv]
\|(\cL_v|_{N_v\cM_1})^{-1}(w)\| \le \ga\|w\|,\quad
 \text{for all $v\in\cM_1$ and $w\in N_v\cM_1$}.
\eeq
Next we define the operator $T(r,v,\cdot):N_v\cM_1\to N_v\cM_1$ by
\[
\begin{aligned}
T(r,v,w)
 &= -(\cL_v|_{N_v\cM_1})^{-1}\circ(Id-P_v)\Big[\nabla\Psi_0(v+w) - \nabla^2\Psi_0(v)[w]\\ &\hspace{1cm}
  +\frac{2\pi rN}{N-1}(Id-\Delta)^{-1}\big(\pa_1F(r\wh{v}+r\wh{w})\big)\Big].
\end{aligned}
\]
Then $w\in N_v\cM_1$ solving the second equation in \eqref{eq:LS2} is equivalent to the fixed point equation $w=T(r,v,w)$. In the following, we will prove that for $r\sim 0$ and $v\in \cM_1^R$ arbitrary, $T(r,v,\cdot)$ is a contraction on a suitable neighborhood of $0$ in $N_v\cM_1$.

{\sc Step 2.}\quad We prove that there exist constants $0<r_2<r_1=\frac{\rho}{1+\kappa\de}$
and $0<\de_1<\de$, such that for any $|r|\le\min\{r_2,\frac{\rho}{R}\}$ and
$w,w'\in N_v\cM_1$ with $\|w\|,\|w'\|\le\de_1$, there holds
\beq[contr-1]
\|T(r,v,w)-T(r,v,w')\|\leqslant\frac{1}{2}\|w-w'\|.
\eeq
In order to see this we first observe that \eqref{eq:norm-Lv} implies for $w,w'\in N_v\cM_1$ with $\|w\|,\|w'\|\leqslant \delta$, that
\beq[contr-2]
\begin{aligned}
\|T(r,v,w)-T(r,v,w')\|
 &\le \gamma\Big(\big\|\nabla\Psi_0(v+w)-\nabla\Psi_0(v+w')-\nabla^2\Psi_0(v)[w-w']\big\|\\
 &\hspace{1cm}
  +\frac{2\pi rN}{N-1}\big\|(Id-\Delta)^{-1}\big(\pa_1F(r\wh{v}+r\wh{w})
    -\pa_1F(r\wh{v}+r\wh{w'})\big)\big\|\Big)
\end{aligned}
\eeq
Next observe that due to \eqref{eq:Psi_0} there exists a constant $0<\de_1<\de$, such that for any $v\in \cM_1$, $w\in N_v\cM_1$ with $\|w\|\le\de_1$,
\beq[contr-3]
\big\|\nabla^2\Psi_0(v+w)-\nabla^2\Psi_0(v)\big\|\le\frac1{4\ga}
\eeq
in operator norm. Hence, there holds for any $w, w'\in N_v\cM_1$ with $\|w\|,\|w'\|\le\de_1$:
\beq[contr-4]
\begin{aligned}
&\big\|\nabla\Psi_0(v+w)-\nabla\Psi_0(v+w')
-\nabla^2\Psi_0(v)[w-w']\big\|\\
&\hspace{1cm}
 = \left\|\int_0^1\big(\nabla^2\Psi_0\big(v+sw+(1-s)w'\big)
      - \nabla^2\Psi_0(v)\big)[w-w']\,ds\right\|\\
 &\hspace{1cm}
 \le \frac1{4\ga}\|w-w'\|.
\end{aligned}
\eeq
In addition, by the uniform boundness of $|rv(t)|$ and the smoothness of $F$, there exists $0<r_2< r_1$, such that for any $|r|\le\frac{\rho}{R}$ and $|r|\le r_2$,
$\|w\|,\|w'\|\le\de$,
\begin{equation*}
2\pi r\big\|\pa_1F(r\wh{v}+r\wh{w})-\pa_1F(r\wh{v}+r\wh{w'})\big\|_{L^2}
 \le\frac1{4\ga}\|w-w'\|,
\end{equation*}
which implies
\beq[contr-5]
\begin{aligned}
&\frac{2\pi rN}{N-1}\big\|(Id-\De)^{-1}\big(\pa_1F(r\wh{v}+r\wh{w})
  -\pa_1F(r\wh{v}+r\wh{w'})\big)\big\|\\
&\hspace{1cm}
 \le 2\pi r\big\|\pa_1F(r\wh{v}+r\wh{w})-\pa_1F(r\wh{v}+r\wh{w'})\big\|_{L^2}\\
&\hspace{1cm}
 \le \frac1{4\ga}\|w-w'\|.
\end{aligned}
\eeq
Substituting \eqref{eq:contr-4} and \eqref{eq:contr-5} into \eqref{eq:contr-2} yields \eqref{eq:contr-1}.

{\sc Step 3.}\quad We shall verify that there is a constant $0<r_3<r_1$, such that $T(r,v,\cdot)$ maps $\big\{w\in N_v\cM_1: \|w\|\le\de_1$\big\} into itself provided
$|r|\le \min\{r_3,\rho/R\}$.\\
Indeed,
\[
\|T(r,v,w)\|
 \le \ga \cdot \left( \big\|\nabla\Psi_0(v+w) - \nabla^2\Psi_0(v)[w]\big\|
      +\frac{2\pi rN}{N-1}\big\|\pa_1F(r\wh{v}+r\wh{w})\big\|_{L^2}\right).
\]
Similarly, $r\pa_1F(r\wh{v}+r\wh{w})$ converges to $0$ uniformly as $r\to 0$, so there exists $0<r_3<r_1$, such that for $|r|\le\min\{r_3,\rho/R\}$,
\[
\frac{2\pi r N}{N-1}\big\| \pa_1F(r\wh{v}+r\wh{w})\big\|_{L^2}
 \le \frac{\de_1}{4\ga},\quad \text{for all }v\in\cM^R_1,\,\|w\|\le\de_1.
\]
Moreover, \eqref{eq:contr-3} implies
\beq[contr-6]
\begin{aligned}
\big\|\nabla\Psi_0(v+w)-\nabla^2\Psi_0(v)[w]\big\|
 &= \big\|\nabla\Psi_0(v+w)-\nabla\Psi_0(v)-\nabla^2\Psi_0(v)[w]\big\|\\
 &= \left\|\int_0^1\big(\nabla^2\Psi_0(v+sw)-\nabla^2\Psi_0(v)\big)[w]\,ds\right\|\\
 &\le \frac1{4\ga}\|w\|.
\end{aligned}
\eeq
Consequently,
\[
\|T(r,v,w)\|
 \le \ga\cdot\big[\,\frac{1}{4\ga}\|w\|+\frac{\delta_1}{4\ga}\,\big]
 \leq \frac{\de_1}2
\]
as long as $|r|\le \frac{\rho}R$, $|r|\le r_3$ and $\|w\|\le\de_1$.

{\sc Step 4.}\quad Application of the contraction mapping principle\\
Taking $r_0=r_0(\rho,\de):=\min\{r_2,r_3\}$ the contraction mapping theorem applied to $T(r,v,\cdot): \{w\in N_v\cM_1:\|w\|\le \de_1\}$ yields that for each $r\le r_0$, $v\in\cM_1^R$, there exists a unique $w=W_R(r,v)\in N_v\cM_1$ with $\|w\|\le\de_1$ solving \eqref{eq:LS2}. Moreover, $P_v$ is continuously differentiable in $v$, hence $W_R(r,\cdot)$ is also of class $\cC^1$. In addition, since $\Psi_r$ is autonomous and $T_v\cM_1$ is $S^1$-equivariant also $W(r,\cdot)$ is $S^1$-equivariant:
$W_R(r,\theta*v) = \theta*W_R(r,v)$.
Observe that $W_R(r,v)$ is uniquely determined, thus $W_R(r,v)=W_{R'}(r,v)$ if $(r,v)$ lies in the domains of both $\cM_1^R$ and $\cM_1^{R'}$.  Hence we can simply write $W$ instead of $W_R$ and obtain a map
\begin{equation*}
W:\cU = \Big\{(r,v)\in \R\times \cM_1:v=\theta u_a, |r|\le r_0, |ra|\le \rho\Big\}
 \to H^1_{2\pi}(\C)
\end{equation*}
as required.
\end{proof}

Next we prove some properties about the behavior of $W(r,v)$ as $r\to 0$, which are crucial for the proof of the main theorem:

\begin{Lem}\label{lem:est-W}
The following holds uniformly on $\cU$ as $r\to 0$:
\begin{itemize}
\item[a)] $\|W(r,v)\| = O(r)$
\item[b)] $\big\|P_v D_vW(r,v)\big\|_{\cL(T_v\cM_1)}=O(r)$.
\end{itemize}
\end{Lem}

\begin{proof} The inequality \eqref{eq:contr-6} implies
\beq[norm-W]
\|W(r,v)\| = \|T\big(r,v,W(r,v)\big)\|
 \le \frac14\|W(r,v)\|+\frac{2\pi N}{N-1}\ga|r|
      \cdot \big\|\pa_1F(r\wh{v}+r\\wh{W}(r,v))\big\|_{L^2}.
\eeq
Since $\|\pa_1F(r\wh{v}+r\wh{W}(r,v))\|_{L^2}$ is uniformly bounded on $\cU$, there exists a constant $M>0$, such that
\[
2\pi\ga\big\|\pa_1F(r\wh{v}+r\wh{W}(r,v))\big\|_{L^2}
 \le M,\quad\text{for all }(r,v)\in\cU.
\]
This, substituted into \eqref{eq:norm-W}, yields
\[
\|W(r,v)\| \le \frac{4}{3}M|r|, \quad\text{for all }(r,v)\in\cU,
\]
proving a).

Next, let $\{f_i(v)\}_{i=1}^3$ be an orthonormal basis of $T_v\cM_1$ depending smoothly on $v\in\cM_1$. In order to estimate $P_vD_vW(r,v)$ we differentiate the identity
\[
P_v W(r,v)=\sum_{i=1}^3\langle W(r,v),f_i(v)\rangle f_i(v)=0
\]
with respect to $v$. This gives
\[
\sum_{i=1}^3 \big(\langle D_vW(r,v)\phi,f_i(v)\rangle f_i(v)
 + \langle W(r,v),f'_i(v)\phi\rangle f_i(v) + \langle W(r,v),f_i(v)\rangle f'_i(v)\phi\big)
 = 0
\]
for any $\phi\in T_v\cM_1$, and therefore,
\[
P_vD_vW(r,v)\phi
 = -\sum_{i=1}^3 \big(\langle W(r,v),f'_i(v)\phi\rangle f_i(v)
    + \langle W(r,v),f_i(v)\rangle f'_i(v)\phi\big).
\]
The invariance of the tangent spaces along $\cM_1$ under translations and the equivariance with respect to the $S^1$-action imply that $f_i(v)$ and $f'_i(v)$ are uniformly bounded for $v\in\cM_1$. Then together with part a), we obtain
\[
\big\|P_vD_vW(r,v)\big\|_{L(T_v\cM_1)}=o(1) \quad\text{as $r\to 0$ uniformly on $\cU$.}
\]
\end{proof}

It remains to solve
\[
P_v\Big(\nabla\Psi_r\big(v+W(r,v)\big)\Big)=0
\]
for $(r,v)\in\cU$. This can be reformulated as a finite-dimensional variational problem using the function
\[
\psi:\cU\to\R,\quad \psi(r,a)=\psi_r(a):=\Psi_r\big(u_a+W(r,u_a)\big).
\]

\begin{Lem}\label{lem:red-to-psi}
There exists $\wt{r_0}>0$, such that if $a$ is a critical point of $\psi_r$ for some $|r|\le\wt{r_0}$, then $\nabla\Psi_r\big(u_a+W(r,u_a)\big)=0$.
\end{Lem}

\begin{proof}  According to Lemma \ref{lem:est-W} b), there exists
$\wt{r_0}>0$ such that
\beq[norm-DvW]
\big\|P_vD_vW(r,v)\big\|\le \frac12 \quad
\text{for all $(r,v)\in \cU$ with $|r|\le\wt{r_0}$.}
\eeq
Suppose $|r|\le\wt{r_0}$ and $\nabla\psi_r(a)=0$. Then
\beq[nabla-Psi-1]
\big\langle\nabla\Psi_r\big(u_a+W(r,u_a)\big),\,a'+D_vW(r,u_a)a'\big\rangle
=0
\eeq
for any $a'\in\C\subset T_{u_a}\cM_1$.

Since $\Psi_r\big(\theta* u_a+W(r,\theta* u_a)\big)$ is independent of $\theta\in S^1$, differentiating it at $\theta=0$ gives
\beq[nabla-Psi-2]
\big\langle\nabla\Psi_r\big(u_a+W(r,u_a)\big),\,\big(Id+D_vW(r,u_a)\big)\dot{u_0}\big\rangle
=0\,.
\eeq
Combining \eqref{eq:nabla-Psi-1} and \eqref{eq:nabla-Psi-2} we obtain
\[
\big\langle\nabla\Psi_r\big(u_a+W(r,u_a)\big),\big(Id+D_vW(r,u_a)\big)\phi\big\rangle
 = 0
\]
for any $\phi\in T_{u_a}\cM_1$. Moreover, as a consequence of \eqref{eq:norm-DvW} the map $Id+P_{u_a}D_vW(r,u_a)$ is invertible on $T_{u_a}\cM_1$, hence
\begin{equation*}
\big\langle\nabla\Psi_r\big(u_a+W(r,u_a)\big),\phi\big\rangle = 0,\quad
\text{for all $\phi\in T_{u_a}\cM_1,$}.
\end{equation*}
This implies
\begin{equation*}
P_{u_a}\nabla\Psi_r\big(u_a+W(r,u_a)\big) = 0,
\end{equation*}
hence $u_a+W(r,u_a)$ is a critical point of $\Psi_r$.
\end{proof}

Now we make a first order Taylor expansion for $\psi_r$\,.

\begin{Lem}\label{lem:psi-taylor}
There holds
\[
\psi_r(a)= c_0+\frac{4\pi^2 N^2}{N-1} h(ra)+\varphi_r(a)
\]
with $\nabla\varphi_r(a)=o(r)$ as $r\to 0$ uniformly on $\cU$.
\end{Lem}

\begin{proof} We compute
\[
\begin{aligned}
\psi_r(a)
 &=\Psi_r\big(u_a+W(r,u_a)\big)\\
 &=\Psi_r\big(u_a+W(r,u_a)\big) - \Psi_r(u_a)+\Psi_0(u_a) - \frac{4\pi^2}{N-1}F(0)
    + \frac{2\pi}{N-1}\int_0^{2\pi}F(r\wh{u}_a)
\end{aligned}
\]
and
\[
\begin{aligned}
\int_0^{2\pi}F(r\wh{u}_a)\,dt
 &= 2\pi N^2h(ra) \\
 &\hspace{1cm}
    + \int_0^{2\pi} \sum_{\mytop{j,k=1}{j\neq k}}^N
      \left(g(ra+re^{i(t+\frac{2\pi(j-1)}{N})},ra+re^{i(t+\frac{2\pi(k-1)}{N})})
      - g(ra,ra)\right)\,dt\\
 &\hspace{1cm}
    + \int_0^{2\pi} \sum_{k=1}^N\left(h(ra+re^{i(t+\frac{2\pi(k-1)}{N})})-h(ra)\right)\,dt\,.
\end{aligned}
\]
Setting $c_0 := \Psi_0(u_a) - \frac{4\pi^2}{N-1}F(0)$ and
\[
\begin{aligned}
\varphi_r(a)
 &:= \frac{2\pi}{N-1} \int_0^{2\pi} \sum_{\mytop{j,k=1}{j\neq k}}^N
      \left(g(ra+re^{i(t+\frac{2\pi(j-1)}{N})},ra+re^{i(t+\frac{2\pi(k-1)}{N})})
       - g(ra,ra)\right)\,dt\\
 &\hspace{1cm}
      + \frac{2\pi}{N-1} \int_0^{2\pi} \sum_{k=1}^N
        \left(h(ra+re^{i(t+\frac{2\pi(k-1)}{N})})-h(ra)\right)\,dt\\
 &\hspace{1cm}
      +\Psi_r\big(u_a+W(r,u_a)\big)-\Psi_r(u_a),
\end{aligned}
\]
we have $\psi_r(a)= c_0+\frac{4\pi^2 N^2}{N-1} h(ra)+\varphi_r(a)$. Now
\[
\begin{aligned}
\big|\nabla\varphi_r(a)\big|
 &\le \frac{2\pi r}{N-1}\int_0^{2\pi}\sum_{\mytop{j,k=1}{j\neq k}}^N
       \left|\pa_1g(ra+re^{i(t+\frac{2\pi(j-1)}{N})},
             ra+re^{i(t+\frac{2\pi(k-1)}{N})})-\pa_1g(ra,ra)\right|\,dt\\
 &\hspace{1cm}
      + \frac{2\pi r}{N-1}\int_0^{2\pi}\sum_{\mytop{j,k=1}{j\neq k}}^N
         \left|\partial_2g(ra+re^{i(t+\frac{2\pi(j-1)}{N})},
           ra+re^{i(t+\frac{2\pi(k-1)}{N})})-\pa_2g(ra,ra)\right|\,dt\\
 &\hspace{1cm}
      + \frac{2\pi r}{N-1}\int_0^{2\pi} \sum_{k=1}^N
         \left|h'(ra+re^{i(t+\frac{2\pi(k-1)}{N})})-h'(ra)\right|\,dt\\
 &\hspace{1cm}
      +\left|\nabla_a\Psi_r\big(u_a+W(r,u_a)\big)-\nabla_a\Psi_r(u_a)\right|\\
 &= o(r) + \left|\nabla_a\Psi_r\big(u_a+W(r,u_a)\big)-\nabla_a\Psi_r(u_a)\right|
\end{aligned}
\]
because $g$ is of class $\cC^1$ and $|ra|$ is uniformly bounded on $\cU$.

Moreover, since $\Psi'_r\big(u_a+W(r,u_a)\big)\in T_{u_a}\cM_1$ and
$D_vW(r,u_a)a'\in N_{u_a}\cM_1$ for any $a'\in \C$, we deduce for $r\to0$:
\[
\begin{aligned}
&\nabla_a\Psi_r\big(u_a+W(r,u_a)\big)[a']-\nabla_a\Psi_r(u_a)[a']\\
&\hspace{1cm}
 = \Psi'_r\big(u_a+W(r,u_a)\big)\big[a'+D_vW(r,u_a)a'\big]-\Psi'_r(u_a)[a']\\
&\hspace{1cm}
 = \Psi'_r\big(u_a+W(r,u_a)\big)[a']-\Psi'_r(u_a)[a']\\
&\hspace{1cm}
 = \Psi'_0\big(u_a+W(r,u_a)\big)[a']-\Psi'_0(u_a)[a']\\
&\hspace{2cm}
  +\frac{2\pi r N}{N-1} \int_0^{2\pi}
    \big(\pa_1F(r\wh{u}_a+r\wh{W}(r,u_a))-\pa_1F(r\wh{u}_a)\big)[a']\,dt\\
&\hspace{1cm}
 = \Psi'_0\big(u_0+W(r,u_a)\big)[a']-\Psi'_0(u_0)[a'] + o(r)\cdot|a'|\\
&\hspace{1cm}
 = \Psi''_0(u_0)\big[a',W(r,u_a)\big] + o(\|W(r,u_a)\|)\cdot|a'|+o(r)\cdot|a'|\\
&\hspace{1cm}
 = o(r)\cdot|a'|,
\end{aligned}\]
uniformly on $\cU$. Here we applied $a'\in\textrm{Ker}\,\Psi_0''(u_0)$ and Lemma \ref{lem:est-W} a). Summarizing we have proved that $\nabla\varphi_r(a)=o(r)$.
\end{proof}

\section{Proof of Theorem \ref{thm:main}}\label{sec:proof}
Suppose $0$ is an isolated stable critical point of $h$ with $h(0)=c$. For any fixed $0<\eps<\rho$, we can choose a Gromoll-Meyer pair $(B,B^-)$ for $0$ of $h$ such that $B\subset B_\eps(0)$ and
$H_*(B,B^-)\cong H^*(h^c,h^c\setminus\{0\})\neq 0$; see \cite{chang:1993, gromoll-meyer:1969}. In particular, $\pa B\subset M^1\cup\cdots\cup M^k$ is contained in a finite union of submanifolds $M^j= (g^j)^{-1}(0)$, where $g^j\in\cC^1(\C,\R)$ with $0$ being a regular value, and $\nabla g^j(a)$ being the exterior normal to $B$ at $a\in\pa B$. By definition there holds
\beq[nabla-h]
\big\langle \nabla h(a),\nabla g^j(a)\big\rangle_{\R^2} \ne 0,\quad
\text{for all $a\in\pa B\cap M^j$, $j=1,\dots,k$,}
\eeq
and
\beq[nabla-h-exit]
\big\langle \nabla h(a), \nabla g^j(a)\big\rangle_{\R^2} < 0,\quad
\text{if, and only if, $a\in B^-\cap M^j$, $j=1,\dots,k$.}
\eeq
Now we scale these sets and functions as
\[
B_r:=\frac{1}{r}B,\quad B_r^-:=\frac{1}{r}B^-,\quad
M^j_r:=\frac{1}{r}M^j,\quad g^j_r(a):=g^j(ra),
\]
so that
$\pa B_r\subset M^1_r\cup\dots\cup M^k_r=(g^1_r)^{-1}(0)\cup\cdots\cup (g^k_r)^{-1}(0) $.

We consider only $|r|\le\min \{r_0, \wt{r_0}\}$ so that the lemmas from Section~\ref{sec:reduction} make sense. Lemma~\ref{lem:psi-taylor} implies for
$a\in\pa B_r\cap M^j_r$, i.~e.\ $ra\in\pa B\cap M^j$:
\[
\begin{aligned}
\big\langle\nabla\psi_r(a), \nabla g^j_r(a) \rangle_{\R^2}
 &= \left\langle \frac{4\pi^2r N^2}{N-1}\nabla h(ra)+\nabla\varphi_r(a),
     \nabla g^j_r(a)\right\rangle_{\R^2}\\
 &= \frac{4\pi^2r ^2N^2}{N-1}\big\langle \nabla h(ra),\nabla g^j(ra)\big\rangle_{\R^2}
     + r\big\langle\nabla\varphi_r(a), \nabla g^j(ra)\big\rangle_{\R^2}\\
 &= \frac{4\pi^2r ^2N^2}{N-1}\big\langle \nabla h(ra),\nabla g^j(ra)\big\rangle_{\R^2}
     + o(r^2)
\end{aligned}
\]
as $r\to0$. Using the compactness of $\pa B$ and \eqref{eq:nabla-h} we see that
\[
\big\langle\nabla\psi_r(a), \nabla g^j_r(a) \big\rangle\neq 0 \quad
\text{for all $a\in\pa B_r\cap M^j_r$, $j=1,\dots,k$.}
\]
for $|r|>0$ small enough. This implies for $|r|>0$ sufficiently small, that $B_r$ is an isolating neighborhood for the negative gradient flow of $\psi_r$, and as a consequence of \eqref{eq:nabla-h-exit} the exit set is $B_r^-$. Since $H_*(B_r,B_r^-)\cong H_*(B,B^-)\neq 0$, there exists a critical point $a_r$ of $\psi_r$ in $B_r$. Then $u_{a_r}+W(r,u_{a_r})$ is a critical point of $\Psi_r$. Rescaling back, we obtain a $T_r$-periodic solution
\begin{equation*}
z^r(t)
 = r\big(\wh{u}_{a_r}+\wh{W}(r,u_{a_r})\big)\left(\frac{2\pi}{T_r}t\right)
 = r\left(
     \begin{array}{cc}
      a_r+e^{i\frac{2\pi}{T_r}t}\\
      a_r+e^{i(\frac{2\pi}{T_r}t+\frac{2\pi}{N})}\\
      \vdots\\
      a_r+e^{i(\frac{2\pi}{T_r}t+\frac{2\pi(N-1)}{N})}
     \end{array}
    \right)
   +r\wh{W}(r,U_{a_r})\left(\frac{2\pi}{T_r}t\right)
\end{equation*}
of (HS) in $B\subset B_\eps(0)$, proving Theorem 1.1.

{\sc Address of the authors:}\\[1em]
\parbox{8cm}{Thomas Bartsch, Qianhui Dai\\
 Mathematisches Institut\\
 Universit\"at Giessen\\
 Arndtstr.\ 2\\
 35392 Giessen\\
 Germany\\
 Thomas.Bartsch@math.uni-giessen.de\\
 qhdai910@gmail.com}

\end{document}